\documentclass[a4paper,11pt,reqno]{amsart}
\usepackage{amsmath}
\usepackage[T1]{fontenc}
\usepackage{amssymb}
\usepackage{amsthm}
\usepackage{color}
\usepackage[lmargin=2.5 cm,rmargin=2.5 cm,tmargin=3.5cm,bmargin=2.5cm,
paper=a4paper]{geometry}

\newcommand{\clb}{\color{blue}}

\newcommand{\Ab}{\mathbf A}

\newcommand{\Ap}{\Ab_{\rm van}}
\newcommand{\Fb}{\mathbf F}

\newcommand{\R}{\mathbb R}

\newcommand{\E}{\mathrm{E}_{\rm gs}(\kappa, h_{\rm ex};B_0)}

\newcommand{\er}{\mathfrak{e}_{\rm gs}}

\DeclareMathOperator{\curl}{curl}

\newtheorem{thm}{Theorem}[section]
\newtheorem{prop}[thm]{Proposition}
\newtheorem{lem}[thm]{Lemma}

\newtheorem{theorem}[thm]{Theorem}
\newtheorem{ass}[thm]{Assumption}

\theoremstyle{remark}

\newtheorem{rem}[thm]{Remark}

\numberwithin{equation}{section}

\title[From constant to non-degenerately vanishing magnetic fields]
{From constant to non-degenerately vanishing magnetic fields in superconductivity}

\author{Bernard Helffer}
\author{Ayman Kachmar}
\address[B. Helffer]{Laboratoire de Math\'ematiques, Universit\'e de Paris-Sud 11, B\^at 425, 91405 Orsay, France and  Laboratoire Jean Leray (Universit\'e de Nantes)}
\email{bernard.helffer@math.u-psud.fr}
\address[A. Kachmar]{Department of Mathematics, Lebanese University, Hadat, Lebanon}
\email{ayman.kashmar@gmail.com}
\date{\today}

\begin{document}

\begin{abstract}
We explore the relationship between two  reference functions arising
in the analysis of the Ginzburg-Landau functional. The first
function describes the distribution of superconductivity in a
type~II superconductor subjected to a constant magnetic field. The
second function describes the distribution of superconductivity in a
type~II superconductor  submitted to a variable magnetic field that
vanishes non-degenerately along a smooth curve.
\end{abstract}

\maketitle 

\section{Introduction}\label{hc2-sec:int}

The Ginzburg-Landau functional is a celebrated phenomenological
model that describes the response of a superconductor to a magnetic
field \cite{dG}. In non-dimensional units, the functional is defined
as follows,
\begin{equation}\label{eq-3D-GLf}
\mathcal E(\psi,\Ab)=\int_\Omega \left(|(\nabla-i\Ab)\psi|^2-\kappa^2|\psi|^2+\frac{\kappa^2}2|\psi|^4+|\curl\Ab-h_{\rm ex}B_0|^2\right)\,dx\,,
\end{equation}
where:
\begin{itemize}
\item $\Omega\subset\R^2$ is an open, bounded and simply connected
set with a smooth boundary\,; $\Omega$ is the cross section of a
cylindrical superconducting sample placed vertically\,;
\item $(\psi,\Ab)\in H^1(\Omega;\mathbb C)\times H^1(\Omega;\mathbb
R^2)$ describe the state of superconductivity as follows:
$|\psi|$ measures the density of the superconducting Cooper
pairs and $\curl\Ab$ measures the induced magnetic field in the
sample\,;
\item $\kappa>0$ is the Ginzburg-Landau parameter, a material characteristic
of the sample\,;
\item $h_{\rm ex}>0$ measures the intensity of the applied magnetic
field\,;
\item $B_0$ is a smooth function defined in $\overline\Omega$. The
applied magnetic field is $h_{\rm ex}B_0 \vec{e}$, where $\vec{
e}=(0,0,1)$.
\end{itemize}
We introduce the ground state energy of the functional in
\eqref{eq-3D-GLf} as follows,
\begin{equation}\label{eq-gse}
\E=\inf\{\mathcal E(\psi,\Ab)~:~(\psi,\Ab)\in H^1(\Omega;\mathbb
C)\times H^1(\Omega;\mathbb R^2)\}\,.
\end{equation}
In physical terms, \eqref{eq-gse} describes the energy of a type~II
superconductor submitted to a possibly non-constant magnetic field of intensity
$h_{\rm ex}|B_0|$.

The behavior of the ground state energy in \eqref{eq-gse} strongly
depends on the values of $\kappa$ and $h_{\rm ex}$. This is the
subject of a vast mathematical literature. In the two monographs
\cite{FH-b, SaSe}, a survey of many important results regarding the
behavior of $\E$ is given. The results are valid when $h_{\rm
ex}=h_{\rm ex}(\kappa)$ is a function of $\kappa$ and
$\kappa\to +\infty$.

Let us recall two important results regarding the ground state
energy in \eqref{eq-gse}. The first result is obtained in
\cite{SS02} and says, if $b\in(0,1]$ is a constant,  $h_{\rm
ex}=b\kappa^2$ and $B_0=1$, then
\begin{equation}\label{eq:gse-SS}
\E=g(b)|\Omega|\kappa^2+o(\kappa^2)\quad(\kappa\to +\infty)\,,
\end{equation}
where $g(b)$ is a constant that will be defined in \eqref{eq:g}
below.

The second result is given in \cite{HK} and valid under the
following assumption on the function $B_0$.
\begin{ass}\label{ass-B0}
Suppose that $B_0:\overline{\Omega}\to\R$ is a smooth function
satisfying
\begin{itemize}
\item $|B_0|+|\nabla B_0|\geq c$ in $\overline{\Omega}$, where $c>0$
is a constant\,;
\item $\Gamma=\{x\in\overline{\Omega}~:~B_0(x)=0\}$ is the union of
a finite number of smooth curves\,;
\item $\Gamma\cap\partial\Omega$ is a finite set\,.
\end{itemize}
\end{ass}
Under these assumptions on $B_0$, if $b>0$ is a constant and $h_{\rm
ex}=b\kappa^3$, then,
\begin{equation}\label{eq:gse-HK}
\E=\kappa\left(\int_{\Gamma}\Big(b|\nabla
B_0(x)|\Big)^{1/3}\,E\Big(b|\nabla
B_0(x)|\Big)\,ds(x)\right)+o(\kappa)\,,
\end{equation}
where $E(\cdot)$ is a {\it continuous} function that will be defined
in \eqref{eq:E} below, and $ds$ is the arc-length measure in
$\Gamma$.


In physical terms, \eqref{eq:gse-HK} describes the energy of a
type~II superconductor subjected to a {\it variable} magnetic field
that {\it vanishes} along a {\it smooth curve}.   Such magnetic
fields are of special importance in the analysis of the
Ginzburg-Landau model in  surfaces (see \cite{CL}).

Magnetic fields satisfying Assumption~\ref{ass-B0} have an early
appearance in the literature, for instance in a paper by Montgomery
\cite{M}. Pan and Kwek \cite{PK} study the breakdown of
superconductivity under the Assumption~\ref{ass-B0}. They find a
constant $c_0>0$ such that, if $h_{\rm ex}= b\kappa^3$, $b> c_0$ and
$\kappa$ is sufficiently large, then $\E=0$. Recently, the results
of Pan-Kwek  have been improved in \cite{Att3, Miq}. The discussion
in \cite{HK} proves that the formula in \eqref{eq:gse-HK} is
consistent with the conclusion in \cite{PK} and with Theorem 1.7 in
\cite{Att3}.

As proven in \cite{HK}, the formula in \eqref{eq:gse-HK} continues
to hold when $h_{\rm ex}=b\kappa^3$ and $b=b(\kappa)$
satisfies\,\footnote{ The notation $a(\kappa)\ll b(\kappa)$ means
that $a(\kappa)=\delta(\kappa)b(\kappa)$ and
$\displaystyle\lim_{\kappa\to +\infty}\delta(\kappa)=0\,$.},
\begin{equation}\label{eq:cond-b}
\kappa^{-1/2}\ll b(\kappa)\ll
1\quad(\kappa\to +\infty)\,.\end{equation}
When the condition in \eqref{eq:cond-b} is violated by
allowing\,\footnote{The notation $a(\kappa)\lesssim b(\kappa)$
means that there exists a  constant $c >0$ and  $\kappa_0>0$ such
that, for all $\kappa\geq \kappa_0$, $a(\kappa)\leq c b(\kappa)\,$.}
$$  \kappa^{-1} \ll b(\kappa)\lesssim  \kappa^{-1/2} \quad(\kappa\to +\infty)\,$$ then the formula in
\eqref{eq:gse-HK} is replaced with (see \cite{HK}),
\begin{equation}\label{eq:gse-HK'}
\E=\kappa^2\int_\Omega g\big(b (\kappa) \, \kappa\, |B_0(x)|\big)\,dx+o\left(b(\kappa) ^{-1}\kappa\right)\,.
\end{equation}
 Note that \eqref{eq:gse-HK'} is still true  for lower values
of the external field but with a different expression for the
remainder term (see \cite{Att,Att2}).

The comparison of the formulas in \eqref{eq:gse-HK} and
\eqref{eq:gse-HK'} at the  border  regime\footnote{The notation
$a(\kappa)\approx  b(\kappa)$ means that $a(\kappa)\lesssim
b(\kappa)$ and $b(\kappa)\lesssim a(\kappa)$.}
$$
b(\kappa)\approx \kappa^{-1/2}$$  suggests that there might
exist a relation between the two reference functions $g(\cdot)$ and
$E(\cdot)$. This paper confirms the existence of such a
relationship.

The two  functions $g(\cdot)$ and $E(\cdot)$ are defined via
simplified versions of the functional in \eqref{eq-3D-GLf}. As we
shall see, $g(\cdot)$ will be defined via a constant magnetic field,
while, for $E(\cdot)$, this will be via a magnetic field that
vanishes along a line.

Let us recall the definition of the function $g(\cdot)$.
Consider $b\in\,(0,+\infty)$, $r>0\,$, and
$Q_r=\,(-r/2,r/2)\,\times\,(-r/2,r/2)$\,. Define the functional,
\begin{equation}\label{eq:rGL}
F_{b,Q_r}(u)=\int_{Q_r}\left(b|(\nabla-i\Ab_0)u|^2-|u|^2+\frac{1}2|u|^4\right)\,dx\,,
\quad \mbox{ for } u\in H^1(Q_r)\,.
\end{equation}
Here, $\Ab_0$ is the magnetic potential,
\begin{equation}\label{eq:A0}
\Ab_0(x)=\frac12(-x_2,x_1)\,,\quad \mbox{ for } x=(x_1,x_2)\in \R^2\,.
\end{equation}
Define the two Dirichlet and Neumann ground state energies,
\begin{align}
&e_D(b,r)=\inf\{F_{b,Q_r}(u)~:~u\in H^1_0(Q_r)\}\,,\label{eq:eD}\\
&e_N(b,r)=\inf\{F_{b,Q_r}(u)~:~u\in H^1(Q_r)\}\,.\label{eq:eN}
\end{align}
Thanks to \cite{Att, FK-cpde, SS02}, $g(\cdot)$ may be defined as
follows,
\begin{equation}\label{eq:g}
\forall~b>0\,,\quad g(b)=\lim_{r\to\infty}\frac{e_D(b,r)}{|Q_r|}=\lim_{r\to\infty}\frac{e_N(b,r)}{|Q_r|}\,,
\end{equation}
where  $|Q_r|$ denotes the area of $Q_r$ ($|Q_r|=r^2$).\\
 Moreover  the
function $g(\cdot)$ is a  non decreasing continuous function such
that
\begin{equation} \label{propg}
g(0)=-\frac12 \mbox{ and } g(b)=0 \mbox{ when } b\geq 1\,.
\end{equation}

Now we   introduce the function $E(\cdot)$.

Let $L>0$, $R>0$, $\mathcal S_R=(-R/2,R/2)\times \R$ and
\begin{equation}\label{eq-A-app}
\Ap(x)=\Big(-\frac{x_2^2}2,0\Big)\,,\quad
\mbox{ for } x=(x_1,x_2)\in\R^2\,.\end{equation} Notice that $\Ap$
is a magnetic potential generating  the magnetic field
\begin{equation}\label{eq-B-app}
B_{\rm van}(x)=\curl\Ap=x_2\,,
\end{equation}
which  vanishes along the $x_2$-axis.

Consider the functional
\begin{equation}\label{eq-gs-er''}
\mathcal E_{L,R}(u)=\int_{\mathcal S_R}\left(|(\nabla-i\Ap)u|^2-L^{-2/3}|u|^2+\frac{L^{-2/3}}{2}|u|^4\right)\,dx\,,
\end{equation}
and the ground state energy
\begin{equation}\label{eq-gs-er'}
\er(L;R)=\inf\{\mathcal E_{L,R}(u)~:~u\in H^1_{{\rm mag},0}(\mathcal S_R)\}\,,
\end{equation}
where
\begin{equation}\label{eq:space-mag}
H^1_{{\rm mag},0}(\mathcal S_R)=\{u\in L^2(\mathcal S_R)~:~(\nabla-i\Ap)u\in L^2(\mathcal S_R)\quad{\rm and}\quad u=0~{\rm on~}\partial\mathcal S_R\}\,.
\end{equation}
Thanks to \cite{HK}, we may define $E(\cdot)$ as follows,
\begin{equation}\label{eq:E}
E(L)=\lim_{R\to\infty}\frac{\er(L;R)}{R}\,.
\end{equation}

In this paper, we obtain a relationship between the functions
$E(\cdot)$ and $g(\cdot)$:

\begin{thm}\label{thm:HK}
Let $g(\cdot)$ and $E(\cdot)$ be as in \eqref{eq:g} and \eqref{eq:E}
respectively. It holds,
$$E(L)=2L^{-4/3}\int_0^1g(b)\,db+o\big(L^{-4/3}\big)\quad{\rm as
~}L\to0_+\,.$$
\end{thm}

As a consequence of Theorem~\ref{thm:HK} and the co-area formula, we
obtain:

\begin{thm}\label{prop:HK+}
Suppose that the function $B_0$ satisfies Assumption~\ref{ass-B0}
and
$$\kappa^{-1}\ll b(\kappa)\ll 1\,.$$
Let $g(\cdot)$ and $E(\cdot)$ be the energies introduced in
\eqref{eq:gse-SS} and \eqref{eq:gse-HK'} respectively. It holds,
\begin{multline*}
\int_\Omega g\big(b (\kappa) \, \kappa\, |B_0(x)|\big)\,dx
\\=\kappa^{-1}\int_{\Gamma}\Big(b(\kappa)|\nabla
B_0(x)|\Big)^{1/3}\,E\Big(b(\kappa)|\nabla B_0(x)|\Big)\,ds(x)+
o\big(b(\kappa)^{-1}\kappa^{-1}\big)\big)\,,\quad(\kappa\to
+\infty)\,.
\end{multline*}
\end{thm}

This yields the following improvement of the main result in
\cite{HK}:

\begin{theorem}\label{thm:HK+}
Suppose that Assumption~\ref{ass-B0} holds and
$$h_{\rm ex}=b(\kappa)\kappa^3\,,\quad \kappa^{-1}\ll
b(\kappa)\lesssim 1\quad(\kappa\to +\infty)\,.$$ The ground state energy in \eqref{eq:gse-HK} satisfies,
$$\E=\kappa\int_\Gamma\Big(b(\kappa)|\nabla
B_0(x)|\Big)^{1/3}\,E\Big(b(\kappa)|\nabla
B_0(x)|\Big)\,ds(x)+o\big(b(\kappa)^{-1}\kappa\big)\,.\quad(\kappa\to +\infty)\,.$$
\end{theorem}

The rest of the paper is devoted to the proof of
Theorems~\ref{thm:HK} and \ref{prop:HK+}.  Note that, along the
proof of Theorem~\ref{thm:HK}, we provide explicit estimates of the
remainder terms (see Theorems~\ref{thm:lb} and \ref{thm:ub}).

\section{Preliminaries}

In this section, we collect useful results regarding the two
functionals in \eqref{eq:gse-SS} and \eqref{eq:gse-HK}.

For the functional in \eqref{eq:gse-SS} and the corresponding ground
state energies in \eqref{eq:eD} and \eqref{eq:eN}, the following
results are given in \cite{Att2,FK-cpde}:

\begin{prop}\label{prop:FK}~
\begin{enumerate}
\item There exist  minimizers of the ground state energies in
\eqref{eq:eD} and \eqref{eq:eN}.
\item For all $r>0$ and $b>0$, a minimizer $u_{b,r}$ of
\eqref{eq:eD} or \eqref{eq:eN} satisfies
$$|u_{b,r}|\leq 1\quad{~\rm in~}Q_r\,.$$
\item For all $r>0$ and $b>0$, $e_D(b,R)\geq e_N(b,R)$.
\item For all $r>0$ and $b\geq 1$, $e_D(b,r)=0$.
\item  There exists a constant $C>0$ such that, for all
$ b>0$ and $ r\geq 1$, then
\begin{equation}\label{eq:2.1}
e_N(b,R)\geq e_D(b,r)-Cr\sqrt{b}\,.
\end{equation}
\item There exists a  constant $C$ such that, for all
$r\geq 1$ and  $b\in(0,1)$,
\begin{equation}\label{eq:g'}
  g(b)\leq\frac{e_D(b,r)}{|Q_r|}\leq g(b)+C\frac{\sqrt{b}}{r}\,.
\end{equation}
\end{enumerate}
\end{prop}

\begin{rem}\label{rem:ext(6)}
The estimate in \eqref{eq:g'} continues to hold when $b\geq
1$, since in this case $g(b)=0$ and $e_D(b,r)=0$.
\end{rem}

\begin{rem}\label{rem:proof(5)}
Let us mention that Inequality \eqref{eq:2.1} is proved in
\cite[Prop.~2.2]{Att2} for $0<b<1$ and can be easily extended for
$b=1$. For $b\geq 1$, we have, $e_D(b,R)=0$, and by a simple
comparison argument,
$$e_N(b,r)\geq e_N(1,r)\geq e_D(1,r)-Cr=e_D(b,r)-Cr\geq e_D(b,r)-Cr\sqrt{b}\,.$$
\end{rem}

\begin{rem}\label{rem:eN}
  We  recall the following simple consequence of  the
assertions (3)-(6) in Proposition~\ref{prop:FK}. Knowing that
$g(b)=0$ for all $b\geq 1$, we may find  a constant $C>0$ such that,
for all $b>0$ and $r\geq 1$,
$$\frac{e_N(b,r)}{|Q_r|}\geq g(b)-C\frac{\sqrt{b}}r\,.$$
\end{rem}

The next lemma indicates a regime where the Neumann energy in
\eqref{eq:eN} vanishes.

\begin{lem}\label{lem:eN=0}
There exists a constant $r_0>0$ such that, for all $r\geq r_0$ and
$b\geq r_0$,
$$e_N(b,r)=0\,.$$
\end{lem}
\begin{proof}
We have the trivial upper bound, valid for all $b>0$ and $r>0$,
$$e_N(b,r)\leq F_{b,Q_r}(0)=0\,.$$
Now we will prove that $e_N(b,r)\geq 0$ for sufficiently large
values of $b$ and $r$. Let $u$ be an arbitrary function in
$H^1(Q_r)$.\\
 We apply a rescaling to obtain,
\begin{equation}\label{eq:eigenvalue}
\int_{Q_r}|(\nabla-i\Ab_0)u|^2\,dx=r^4\int_{Q_1}|(r^{-2}\nabla-i\Ab_0)v|^2\,dy\,,
\end{equation}
where $$v(y)=u(ry)\,.$$
 For every $h>0$, we  introduce the following
ground state eigenvalue,
$$ \mu_1(h)=\inf_{\substack{v\in
H^1(Q_1)\\v\not=0}}\frac{\displaystyle\int_{Q_1}|(h\nabla-i\Ab_0)v|^2\,dy}{\displaystyle\int_{Q_1}|v|^2\,dy}\,.$$
It is a known fact that (see  \cite{Bon, Pan, FH-b}),
$$\lim_{h\to0_+}\frac{\mu_1(h)}{h}=\Theta_1\,,$$
where $\Theta_1\in(0,1)$ is a universal constant.

In that way, we get a constant $r_1>0$ such that, for all $r\geq
r_1$, we infer from \eqref{eq:eigenvalue},
$$
 \int_{Q_r}|(\nabla-i\Ab_0)u|^2\,dx\geq\frac{\Theta_1}2\int_{Q_1}|v (y)|^2\,r^2dy=\frac{\Theta_1}2\int_{Q_r}|u (x)|^2\,dx \,.
$$
 We insert this into the expression of  $F_{b,Q_r}(u)$ to get, for
all $r\geq r_1$ and $b>0$,
$$F_{b,Q_{r}}\geq \int_{Q_{r}}
\left(b\frac{\Theta_1}2-1\right)|u|^2\,dx\,.$$ Let
$r_0=\max(r_1,2\Theta_1^{-1})$. Clearly, for all $r\geq r_0$, $b\geq
r_0$ and $u\in H^1(Q_r)$, $F_{b,Q_r}(u)\geq 0$. Consequently,
$e_N(b,r)\geq 0$.
\end{proof}

The functional in \eqref{eq:gse-HK} is studied in \cite{HK}. In
particular, the following results were obtained:

\begin{prop}\label{prop:HK}~
\begin{enumerate}
\item For all $L>0$ and $R>0$, there exists a minimizer
$\varphi_{L,R}$ of \eqref{eq-gs-er'}.
\item The function $\varphi_{L,R}$ satisfies
$$|\varphi_{L,R}|\leq 1\quad{\rm in~}\mathcal S_R\,.$$
\item There exists a constant $C>0$ such that, for all $L>0$ and
$R>0$,
\begin{equation}\label{eq:ub-u}
\int_{\mathcal S_R}|\varphi_{L,R}(x)|^2\,dx\leq CL^{-2/3}R\,.
\end{equation}
\item For all $ L>0$ and $R>0$,
\begin{equation}\label{eq:lb-er}
E(L)\leq \frac{\er(L;R)}{R} \,.
\end{equation}
\item There exists a constant $C>0$ such that, for all $L>0$ and $\ R\geq 4$,
\begin{equation}\label{eq:ub-er}
 \frac{\er(L;R)}{R}\leq E(L)+C\left(1+L^{-2/3}\right)R^{-2/3}\,.
\end{equation}
\end{enumerate}
\end{prop}

\section{Proof of Theorem~\ref{thm:HK}: Lower bound}

The aim of this section is to prove the lower bound in Theorem
\ref{thm:HK}.  Note that  the lower bound below is with a
better remainder term.

\begin{thm}\label{thm:lb}
There exist two constants $L_0>0$ and $C>0$ such that, for all
$L\in(0,L_0)$,
$$E(L)\geq 2L^{-4/3}\int_0^1g(b)\,db-CL^{-1}\,,$$
where $E(\cdot)$ and $g(\cdot)$ are the energies introduced in
\eqref{eq:E} and \eqref{eq:g} respectively.
\end{thm}

The proof of Theorem~\ref{thm:lb} relies  on  the
following lemma:

\begin{lem}\label{lem:lb}
Let $M>0$. There exist two  constants $C>0$ and $A_0\geq 4$ such
that, if
$$A\geq A_0\,,\quad  R\geq 1,\quad 0<L\leq A^{-3/2}\,,
\quad u\in H^1(\mathcal S_R)\,,$$
$$\|u\|_\infty\leq 1\quad{\rm and}\quad\int_{\mathcal S_R}|u|^2\,dx\leq ML^{-2/3}R\,,$$
then
$$
\int_{\mathcal S_R\cap\{|x_2|\geq A\}}\left(|(\nabla-i\Ap)u|^2-L^{-2/3}|u|^2+\frac{L^{-2/3}}2|u|^4\right)\,dx\geq 2RL^{-4/3}\int_0^1g(b)\,d b-CRL^{-1}\,.
$$
\end{lem}
\begin{proof}
Let $L\in(0,1)$, $A>0$ and  $R$ and $u$ satisfy the assumptions in
Lemma~\ref{lem:lb}.  If $\mathcal D\subset\mathcal S_R$, then we use
the notation
\begin{equation}\label{eq-D}
\mathcal E(u;\mathcal D)=\int_{ \mathcal
D}\left(|(\nabla-i\Ap)u|^2-L^{-2/3}|u|^2+ \frac 12\, L^{-2/3}\,
|u|^4\right)\,dx\,.
\end{equation}
We will prove that,
\begin{equation}\label{eq:x2>A}
\mathcal E(u;\mathcal S_R\cap\{x_2\geq A\})\geq RL^{-4/3}\int_0^1g(b)\,d b-CRL^{-1}\,,
\end{equation}
and
\begin{equation}\label{eq:x2<-A}
\mathcal E(u;\mathcal S_R\cap\{x_2\leq -A\})\geq RL^{-4/3}\int_0^1g(b)\,d b-CRL^{-1}\,,
\end{equation}
for some constant $C$ independent of $L$, $R$, $A$, $L$ and $u$.

We will write the detailed proof of \eqref{eq:x2>A}. The proof of
\eqref{eq:x2<-A} is identical.

Let $r_0$ be the {\it universal} constant introduced in
Lemma~\ref{lem:eN=0}. We define $b_0=2\max(1,r_0^2)$. Thanks to
Lemma~\ref{lem:eN=0}, we have,
\begin{equation}\label{eq:eN-rem}
\forall~b\geq \frac{b_0}2\,,\quad\forall~r\geq \sqrt{b_0}\,,\quad e_N(b,r)=0\,,
\end{equation}
where $e_N$ is the Neumann ground state energy introduced in
\eqref{eq:eN}.

We define the constant $A_0=4\sqrt{b_0}$. We introduce $n\in\mathbb
N$ and
$$\ell=n^{-1}R\,.
$$
We will fix a choice of  $n$ later at the end of this proof   such
that (for all $A\geq A_0$),
\begin{equation}\label{eq:cond-A}
 R<n\leq \frac{\sqrt{A}\,R}{2\sqrt{b_0}}\,,
\end{equation}
which ensures that $0<\ell<1$, some $n$ always exists,  and
$$\sqrt{A}\,\ell\geq 2\sqrt{b_0}\,.$$
Let $(Q_{\ell,j})_{j\in\mathcal J}$ be the lattice of squares
generated by
$$Q_\ell=(-R/2,-R/2+\ell)\times(A,A+\ell)\,,$$
and covering $\R^2\setminus\{x_2\leq A\}$.

For every $j\in\mathcal J$, let $c_j=(c_{j,1},c_{j,2})\in\R^2$ be
the center of the square $Q_{\ell,j}$, i.e.
$$ Q_{\ell,j}=(-\ell/2+c_{j,1},\ell/2+c_{j,1})\times(-\ell/2+c_{j,2},\ell/2+c_{j,2})\,.$$
Let $\Ab_0$ be the magnetic potential in \eqref{eq:A0},
$j\in\mathcal J$, $a_j=(a_{j,1},a_{j,2})\in\overline{Q_{\ell,j}}$ be
an arbitrary point  and $$  \mathbf F_j(x_1,x_2)=\Big(
-\frac13(x_2-a_{j,2})^2,\frac13(x_2-a_{j,2})(x_1-a_{j,1})\Big)\,.$$
 Note that, for the sake of simplicity, we omitted the
reference to $\ell$ in the notion of  $c_j$, $a_j$ and $\mathbf
F_j$. It is easy to check that
$$\curl\Ap=\curl\Big(a_{j,2}\Ab_0+\mathbf F_j\Big)\quad{\rm in }\quad Q_{\ell,j}\,.$$
Since the square $Q_{\ell,j}$ is a simply connected domain in
$\R^2$, then there exists a real-valued smooth function $\phi_j$
defined in $Q_{\ell,j}$ such that
$$
\Ap= a_{j,2}\Ab_0+\mathbf F_j-\nabla\phi_j\quad {\rm in}\quad
Q_{\ell,j}\,.$$ Let us define the smooth function
\begin{equation} \label{eq:defphi}
\phi_j(x)=f_j(x)+a_{j,2}\Ab_0(c_j)\cdot x\quad(x\in Q_{\ell,j})\,.
\end{equation}
Now, we have,
\begin{equation}\label{eq:gauge1}
\Ap(x)= a_{j,2}\Ab_0(x-c_j)+\mathbf F_j(x)-\nabla\phi_j(x)\quad {\rm in}\quad
Q_{\ell,j}\,.\end{equation} Thanks to the definition of $\mathbf
F_j$, we have,
\begin{equation} \label{eq:gauge2}
|\mathbf F_j(x)|\leq \ell^2\quad{\rm in}\quad Q_{\ell,j}\,.
\end{equation}
Now, we write the obvious decomposition formula,
\begin{equation}\label{eq:decomp}
\mathcal E(u;\mathcal S_R\cap\{x_2\geq A\})=\sum_{j\in\mathcal J}\mathcal
E(u; Q_{\ell,j})\,.
\end{equation}
We write a lower bound for $\mathcal E(u; Q_{\ell,j})$ when
$j\in\mathcal J$. Recall that, by assumption, for all $j\in\mathcal
J$, $Q_{\ell,j}\subset\{x_2\geq A\}$.
Let $0<\eta<\frac12$. Thanks to \eqref{eq:gauge1}, we may write,
\begin{align*}
\mathcal E(u; Q_{\ell,j})
&=\int_{Q_{\ell,j}}\left(\big|\big(\nabla-i(\Ap+\nabla\phi_j)\big)e^{i\phi_j}u\big|^2-L^{-2/3}
|e^{i\phi_j}u|^2+\frac{L^{-2/3}}2|e^{i\phi_j}u|^4\right)\,dx\\
&\geq \int_{Q_{\ell,j}}\left((1-\eta)\big|\big(\nabla-ia_{j,2}\Ab_0(x-c_j)\big)e^{i\phi_j}u\big|^2-
L^{-2/3}|e^{i\phi_j}u|^2+\frac{L^{-2/3}}2|e^{i\phi_j}u|^4\right)\,dx\\
&
\hskip0.5cm
-4\eta^{-1}\int_{Q_{\ell,j}} |F_j(x)|^2 \,|u|^2\,dx\,.
\end{align*}
Using the bound in \eqref{eq:gauge2}, we get further,
\begin{multline*}
\mathcal E(u; Q_{\ell,j}) \geq
\int_{Q_{\ell,j}}\Bigg((1-\eta)\big|\big(\nabla-ia_{j,2}\Ab_0(x-c_j)\big)e^{i\phi_j}u\big|^2\\-L^{-2/3}|e^{i\phi_j}u|^2+\frac{L^{-2/3}}4|e^{i\phi_j}u|^2\Bigg)\,dx
-C\eta^{-1}\ell^4\int_{Q_{\ell,j}}|u|^2\,dx\,.
\end{multline*}
Recall the definition of the energy in \eqref{eq:eN}. A change of
variable yields,
\begin{equation}\label{eq:en-square}
\mathcal E(u; Q_{\ell,j})
\geq \frac1{L^{2/3}|a_{j,2}|}\,e_N\Big((1-\eta)|a_{j,2}|L^{2/3}\,,\,\sqrt{|a_{j,2}|}\,\ell\Big)-C\eta^{-1}\ell^4\int_{Q_{\ell,j}}|\varphi_{L,R}|^2\,dx\,.
\end{equation}
Let us introduce the two new sets of indices,
\begin{multline*}\widetilde{\mathcal J}=\{j\in\mathcal
J~:~Q_{\ell,j}\cap\{|x_2|\leq
\frac12b_0(1-\eta)^{-1}L^{-2/3}\}\not=\emptyset\}\\\quad{\rm
and}\quad \mathcal J_\infty=\{j\in\mathcal
J~:~Q_{\ell,j}\subset\{|x_2|\geq
\frac12b_0(1-\eta)^{-1}L^{-2/3}\}\}\,.\end{multline*}
 Note
that $\mathcal J=\widetilde{\mathcal J}\cup\mathcal J_\infty$ and we
can decompose every sum over $\mathcal J$ in the following obvious
way
\begin{equation}\label{eq:sum=I+II}
\sum_{j\in\mathcal J}=\sum_{j\in\widetilde{\mathcal
J}}+\sum_{j\in\mathcal J_\infty}\,.
\end{equation}
Furthermore, the
set $\widetilde{\mathcal J}$ is non-empty if $A\leq
b_0(1-\eta)^{-1}L^{-2/3}$. Since $\eta\in(0,\frac12)$ and $b_0\geq
1$, this last condition  is satisfied when $0<L\leq A^{-3/2}$. We
will assume this condition henceforth.

Since $|a_{j,2}|\geq A$ and $b_0\geq 1$, then the condition in
\eqref{eq:cond-A} ensures that
$$\sqrt{|a_{j,2}|}\,\ell\geq 2\sqrt{b_0}> 1\,.$$
 Now, if $j\in\widetilde{\mathcal J}$, then
 we can
 use the lower bound in \eqref{rem:eN}  with $b=
(1-\eta)|a_{j,2}|L^{2/3}  $ and $r=  \sqrt{|a_{j,2}|}\,\ell $  to
write, for a different constant $C>0$,
$$
\mathcal E(u; Q_{\ell,j})
\geq L^{-2/3} \ell^2 \Big(g\big((1-\eta)|a_{j,2}|L^{2/3}\big)- \frac{C}{\ell} \sqrt{1-\eta}\,L^{1/3}\Big)-C\eta^{-1}\ell^4\int_{Q_{\ell,j}}|u|^2\,dx\,.
$$
If $j\in\mathcal J_\infty$, then  $(1-\eta)|a_{j,2}|L^{2/3}\geq
\frac12b_0$ and we can use the identity in \eqref{eq:eN-rem} to
write
$$e_N\Big((1-\eta)|a_{j,2}|L^{2/3}\,,\,\sqrt{|a_{j,2}|}\,\ell\Big)=
0\,.$$
Now  we can infer from \eqref{eq:decomp} the following estimate,
$$
 \mathcal E(u;\mathcal S_R\cap\{x_2\geq A\})\geq L^{-2/3}\sum_{j\in\widetilde{\mathcal J}}
\Big(g\big((1-\eta)|a_{j,2}|L^{2/3}\big)- \frac{C}{\ell}\,
\sqrt{1-\eta}\,L^{1/3}\Big)\ell^2 -C\eta^{-1}\ell^4\int_{\mathcal
S_R}|u|^2\,dx\,.
$$
Using the assumption on the $L^2$-norm of $u$  (see
Lemma~\ref{lem:lb}), we get further,
$$
 \mathcal E(u;\mathcal S_R\cap\{x_2\geq A\})\geq L^{-2/3}\sum_{j\in\widetilde{\mathcal J}}
\Big(g\big((1-\eta)|a_{j,2}|L^{2/3}\big)-\frac{C}{\ell}\, \sqrt{1-\eta}\,L^{1/3}\Big)\ell^2\\
-C\eta^{-1}\ell^4RL^{-2/3}\,.
$$
For any $j \in \mathcal J$, we choose in $\overline{Q_{j,\ell}}$ the previously free point $a_j$ as   $a_j :=
\big(c_{j,1}, c_{j,2} + \frac \ell 2\big).$\\
Since
$g(\cdot)$ is a non decreasing function, this choice yields that,
$$
g\big( (1-\eta) a_{j,2} L^{\frac 23}\big)= \sup_{ t\in (-\frac \ell 2 + c_{j,2}, c_{j,2} + \frac \ell 2) } g\big( (1-\eta) t L^{\frac 23}\big)\,.
$$
In that way,  the sum
$$\ell^2 \sum_{j\in\widetilde{\mathcal J}}
g\big((1-\eta)|a_{j,2}|L^{2/3}\big)$$ is an upper Riemann sum  of
the function $(x_1,x_2) \mapsto   g( (1-\eta) |x_2| L^{\frac 23})$
on $\mathcal D_{L,R}:=\bigcup_{j\in \widetilde{ \mathcal J}}
Q_{\ell,j}$ and
%
%
%
\begin{multline*}
 \mathcal E(u;\mathcal S_R\cap\{x_2\geq A\})\geq L^{-2/3}\int_{\mathcal D_{L,R}}
g\big((1-\eta)|x_2|L^{2/3}\big)\,dx_1dx_2-C(1-\eta)^{-1/2}\,L^{-1}R\\
-C\eta^{-1}\ell^4RL^{-2/3}\,.
\end{multline*}
We now observe that, by definition of $\widetilde{\mathcal J}$ and
$\mathcal J$,
$$ \mathcal D_{L,R}=\bigcup_{j\in\widetilde{\mathcal J}}Q_{\ell,j}\subset\{(x_1,x_2)\in\R^2~:~|x_1|\leq R/2\quad{\rm
and}\quad A< x_2\leq b_0(1-\eta)^{-1}L^{-2/3}+\ell\}\,.$$
 Since $g(\cdot)$ is valued in
$]-\infty,0]$ and $g(b)=0$ for all $b\geq 1$,
 then
$$
\int_{\mathcal D_{L,R}}
g\big((1-\eta)|x_2|L^{2/3}\big)\,dx_1dx_2\geq
\int_{0\leq x_2\leq b_0(1-\eta)^{-1}L^{-2/3}+\ell}\int_{|x_1|\leq R/2}
g\big((1-\eta)|x_2|L^{2/3}\big)\,dx_1dx_2\,,$$ and a simple change of variable yields,
$$
\int_{\mathcal D_{L,R}}
g\big((1-\eta)|x_2|L^{2/3}\big)\,dx_1dx_2\geq
R(1-\eta)^{-1}L^{-2/3}\int_0^1 g(t)\,dt\,.
$$
 Therefore, we have
proved the following lower bound,
$$
 \mathcal E(u;\mathcal S_R\cap\{x_2\geq A\})\geq L^{-4/3}R(1-\eta)^{-1}\int_0^1 g(t)\,dt
-C(1-\eta)^{-1/2}\,L^{-1}R -C\eta^{-1}\ell^4RL^{-2/3}\,.
$$
Now, we choose $n=[R+1]$ where $[\,\cdot\,]$ denotes the integer part.
 In that way, the condition in \eqref{eq:cond-A} is
satisfied for all $R\geq 1$ and $A\geq A_0=4\sqrt{b_0}$. Moreover,
we have the lower bound,
$$
 \mathcal E(u;\mathcal S_R\cap\{x_2\geq A\})\geq 2L^{-4/3}R(1-\eta)^{-1}\int_0^1 g(t)\,dt
-C(1-\eta)^{-1/2}\,L^{-1}R -C\eta^{-1}RL^{-2/3}\,.$$
Now, we choose $\eta=\frac12L^{1/3}$ so that, for all $L\in(0,1)$,
$\eta\in(0,\frac12)$, $\eta^{-1}L^{-2/3}=2L^{-1}$, $\eta
L^{-4/3}=\frac12L^{-1}$ and the lower bound in \eqref{eq:x2>A} is
satisfied.
\end{proof}

\begin{proof}[Proof of Theorem~\ref{thm:lb}]
We use the conclusion in Lemma~\ref{lem:lb} with the following
choices,
$$
R=4\,,\quad A=A_0\,,\quad 0<L\leq L_0:=A^{-3/2}\,,\quad u=\varphi_{L,R}\,,
$$
where $\varphi_{L,R}$ is a
   minimizer of $\mathcal
E_{L,R}$. Notice that, the estimates  in Proposition~\ref{prop:HK}
ensure that the function $u=\varphi_{L,R}$ satisfies the assumptions
in Lemma~\ref{lem:lb}

Thanks to \eqref{eq:ub-er}, we may write,
\begin{equation}\label{eq:1stlb}
E(L)\geq \frac{\mathcal E_{L,R}(\varphi_{L,R})}{R}- C(1+L^{-2/3})\,.
\end{equation}
%
%
By splitting the integral over $\mathcal S_R$ into two parts
$$
\int_{\mathcal S_R}=\int_{\mathcal S_R\cap\{|x_2|\geq A\}}+\int_{\mathcal S_R\cap\{|x_2|\leq A\}}\,,
$$
then using that
$$
|(\nabla-i\Ap)\varphi_{L,R}|^2-L^{-2/3}|\varphi_{L,R}|^2+\frac{L^{-2/3}}{2}|\varphi_{L,R}|^4\geq -L^{-2/3}|\varphi_{L,R}|^2\,,
$$
we get,
\begin{multline*}
\mathcal E_{L,R}(\varphi_{L,R})\geq \int_{\mathcal
S_R\cap\{|x_2|\geq A\}}
\left(|(\nabla-i\Ap)\varphi_{L,R}|^2-L^{-2/3}|\varphi_{L,R}|^2+\frac{L^{-2/3}}2|\varphi_{L,R}|^4\right)\,dx\\
 -\int_{\mathcal S_R\cap\{|x_2|\leq
A\}}L^{-2/3}|\varphi_{L,R}|^2\,dx\,.
\end{multline*}
Now, we use the conclusion in Lemma~\ref{lem:lb} and the bound
$\|\varphi_{L,R}\|_\infty\leq 1$ to write,
$$
\mathcal E_{L,R}(\varphi_{L,R})\geq 2RL^{-4/3}\int_0^1g(b)\,d
b-CRL^{-1} -2ARL^{-2/3}\,.$$ We insert this into  \eqref{eq:1stlb} to finish the proof of Theorem~\ref{thm:lb}.
\end{proof}

\section{Proof of Theorem~\ref{thm:HK}: Upper bound}

The aim of this section is to prove the following upper bound
version of Theorem~\ref{thm:HK}. Note that we provide an explicit
control of the remainder term.

\begin{thm}\label{thm:ub}
There exist two constants $L_0>0$ and $C>0$ such that, for all
$L\in(0,L_0)$,
$$E(L)\leq 2L^{-4/3}\int_0^1g(b)\,db+CL^{-1/3}\,,$$
where $E(\cdot)$ and $g(\cdot)$ are the energies introduced in
\eqref{eq:E} and \eqref{eq:g} respectively.
\end{thm}

The proof of Theorem~\ref{thm:ub} relies on  the following lemma:

\begin{lem}\label{lem:ub}
Let $R\geq 1$, $L>0$,  $\ell\in(0,1)$, $\eta\in(0,1)$,
$c=(c_1,c_2)\in \R^2$ and
$$
 Q_\ell=(-\ell/2+c_1,c_1+\ell/2)\times
(-\ell/2+c_2,c_2+\ell/2)\,.$$
Suppose that
$$Q_\ell\subset\{(x_1,x_2)\in\R^2~:~|x_1|\leq R/2\quad{\rm and}\quad
|x_2|\geq \frac1{\ell^2}\}\,.$$
For all $R\geq 1$, it holds,
\begin{multline*}
\inf\{\mathcal E_{L,R}(w)~:~w\in H^1_0(Q_\ell)\}
\\
\leq
L^{-2/3}\int_{Q_\ell}g\Big((1+\eta)L^{2/3}|x_2|\Big)\,dx_1dx_2+CL^{-2/3}\Big(\ell^{-1}L^{1/3}+\eta^{-1}\ell^4\Big)\ell^2\,,\end{multline*}
where, for all $w\in H^1_0(Q_\ell)$, $\mathcal E_{L,R}(w)$ is
introduced in \eqref{eq-gs-er''} by setting $w=0$ outside $Q_\ell$,
and $C>0$ is a constant independent of $\ell$, $\eta$,  $c$, $L$ and
$R$.
\end{lem}
\begin{proof}
We write the details of the proof when $Q_\ell\subset\{x_2\geq
\ell^{-2}\}$. The case $Q_\ell\subset\{x_2\leq -\ell^{-2}\}$ can be
handled similarly. Let $a=(a_1,a_2)\in \overline{Q_\ell}$. As we did
in the derivation of \eqref{eq:gauge1}, we may define a smooth
function $\phi$ in $Q_\ell$  such that,
\begin{equation}\label{eq:gauge3}
\Ap(x)= a_{2}\Ab_0(x-c)+\mathbf F(x)-\nabla\phi(x)\quad {\rm in}\quad
Q_{\ell}\,,
\end{equation}
and
\begin{equation}\label{eq:gauge4}
|\Fb(x)|\leq C\ell^2\quad{\rm in}~Q_\ell\,,
\end{equation}
where $C>0$ is a universal constant. \\
We introduce the following
three parameters,
\begin{equation}\label{eq:b,r}
\eta\in(0,1)\,,\quad b=a_2(1+\eta)L^{2/3}\,,\quad r=\sqrt{a_2}\,\ell\,.
\end{equation}
Define the following function,
$$u(x)=e^{i\phi(x)}u_{b,r}\big(\sqrt{a_2}\,(x-c)\big)\,,\quad x\in
Q_\ell\,,$$ where $u_{b,r}\in H^1_0(Q_r)$ is a minimizer of the energy $e_D(b,r)$ in \eqref{eq:eD}.

Clearly, $u\in H^1_0(Q_\ell)$. Hence,
$$
\inf\{\mathcal E_{L,R}(w)~:~w\in H^1_0(Q_\ell)\}\leq\mathcal E_{L,R}(u)\,.
$$
Using \eqref{eq:gauge3} and the Cauchy-Schwarz inequality, we
compute the energy of $u$ as follows,
\begin{multline*}
\mathcal E_{L,R}(u)\leq
\int_{Q_\ell}\left((1+\eta)|(\nabla-ia_2\Ab_0(x-c))e^{-i\phi}u|^2-L^{-2/3}|u|^2+\frac{L^{-2/3}}2|u|^4\right)\,dx\\
+4\eta^{-1}\int_{Q_\ell}|\Fb(x)|^2|u|^2\,dx\,.
\end{multline*}
Using \eqref{eq:gauge4}, the bound $|u_{b,r}|\leq 1$, a change of
variable and \eqref{eq:b,r}, we get,
$$\mathcal E_{L,R}(u)\leq
\frac{L^{-2/3}}{a_2}F_{b,r}(u_{b,r})+C\eta^{-1}\ell^6\,,$$ where
$F_{b,r}$ is the functional in \eqref{eq:rGL}. \\
Our choice of
$u_{b,r}$ ensures that,
$$F_{b,r}(u_{b,r})=e_D(b,r).$$
Again,  thanks to the choice of $b$ and $r$ in \eqref{eq:b,r}, we
get,
$$\mathcal E_{L,R}(u)\leq
\frac{L^{-2/3}}{a_2}e_D\Big((1+\eta)a_2L^{2/3},\sqrt{a_2}\,\ell\Big)+C\eta^{-1}\ell^6\,.$$
Now, by the assumption $ Q_\ell\subset\{x_2\geq \ell^{-2}\}$, we
know that $\sqrt{a_2}\,\ell\geq 1$. Thus we may use \eqref{eq:g'} to
write,
\begin{align*}
\mathcal E_{L,R}(u)&\leq
\frac{L^{-2/3}}{a_2}\left(g\big((1+\eta)a_2L^{2/3}\big)+\frac{C\sqrt{1+\eta}\,L^{1/3}}{\ell}\right)(\sqrt{a_2}\ell)^2
+C\eta^{-1}\ell^6\\
&=L^{-2/3}\left(g\big((1+\eta)a_2L^{2/3}\big)+\frac{C\sqrt{1+\eta}\,L^{1/3}}{\ell}\right)\ell^2+C\eta^{-1}\ell^6\,,\end{align*}
which is uniformly true for  $ a \in \overline{Q_\ell}\,$.\\
We now select
$a=\big(c_1,c_{2} -\frac \ell 2\big)\,$.
Since $g(\cdot)$ is a non-decreasing function, then
$$
g\big((1+\eta)a_2L^{2/3}\big)=\inf_{x_2 \in (-\frac \ell 2+c_2, c_2 +\frac \ell 2)} g\big((1+\eta)x_2L^{2/3}\big)\,.$$
This yields,
$$
\ell^2 \, g\big((1+\eta)a_2L^{2/3}\big )\leq
\int_{Q_\ell}g\big((1+\eta)x_2L^{2/3}\big)\,dx_1dx_2\,,
$$
and finishes the proof of Lemma~\ref{lem:ub}.
\end{proof}

\begin{proof}[Proof of Theorem~\ref{thm:ub}]~\\
Let $R=4$, $L\in(0,1)$, $\eta=L$ and $\ell=\frac14$.
Let $(Q_{\ell,j})_{j}$ be the lattice of squares generated by the
square
$$Q=(-R/2,-R/2+\ell)\times(\ell^{-2}, \ell^{-2}+\ell)\,.$$
Define the set of indices
$$\mathcal J=\big\{j~:~Q_{\ell,j}\subset\mathcal S_R\cap\{ x_2\geq\ell^{-2}\}\quad{\rm and}~Q_{\ell,j}\cap\{x_2\leq  (1+\eta)^{-1}L^{-2/3}\}\not=\emptyset\big\}\,.$$
For all $x=(x_1,x_2)\in\R^2$ with $x_2\geq0$, define  $u(x)$ as follows,
$$u(x)=\left\{
\begin{array}{ll}
u_{\ell,j}(x)&{\rm if~}j\in\mathcal J\,,\\
0&{\rm if~}j\not\in\mathcal J\,,
\end{array}
\right.
$$
where $u_{\ell,j}\in H^1_0(Q_{\ell,j})$ is a minimizer of the
following ground state energy
$$\inf\{\mathcal E_{L,R}(w)~:~w\in H^1_0(Q_{\ell,j})\}\,.$$
We extend $u(x)$ in $\{x_2\leq 0\}$ as follows,
$$ u(x)= \bar u(x_1,-x_2)\,,\quad x=(x_1,x_2)\quad {\rm and ~}x_2\leq 0\,.$$
Clearly, $u\in H^1_{\rm mag,0}(\mathcal S_R)$. Notice that,
$$\mathcal E_{L,R}(u)=2\sum_{j\in\mathcal J}\mathcal
E_{L,R}(u_{\ell,j})\,,$$ and for $j\in\mathcal J$, the square
$Q_{\ell,j}$ satisfies the assumption in Lemma~\ref{lem:ub}. We use
Lemma~\ref{lem:ub} to write, \begin{equation}\label{eq:ub} \mathcal
E_{L,R}(u)\leq 2L^{-2/3}\int_{\mathcal
D_\ell}g\Big((1+\eta)L^{2/3}x_2\Big)\,dx_1dx_2+CL^{-1/3}|\mathcal
D_\ell|\,,\end{equation} where the domain $\mathcal D_\ell$ is given
as follows,
$$\mathcal D_\ell=\bigcup_{j\in\mathcal J}\overline{Q_{\ell,j}}\,.$$
Thanks to the definition of the set $\mathcal J$, it is clear that,
$$ \mathcal S_R\cap\{\ell^{-2}\leq x_2\leq
(1+\eta)^{-1}L^{-2/3}\}\subset\mathcal D_\ell  \subset   \mathcal
S_R\cap\{0\leq x_2\leq (1+\eta)^{-1}L^{-2/3}+\ell\}\,.$$
This yields:
$$
|\mathcal D_\ell|=\mathcal O(RL^{-1/3})\,,$$ and (since the function
$g(\cdot)$ is valued in $ [-\frac 12,0]$ and $g(b)=0$ for all $b\geq
1$),
\begin{align*}
\int_{\mathcal
D_\ell}g\Big((1+\eta)L^{2/3}x_2\Big)\,dx_1dx_2&  \leq
\int_{\ell^{-2}}^{(1+\eta)^{-1}L^{-2/3}} \int_{-R/2}^{R/2}g\Big((1+\eta)L^{2/3}x_2\Big)\,dx_1dx_2\\
& =(1+\eta)^{-1}L^{-2/3}R\int_{\ell^{-2}(1+\eta)L^{2/3}}^1g(t)\,dt\\
&\leq (1+\eta)^{-1}L^{-2/3}R\int_{0}^1g(t)\,dt  +\ell^{-2}R\,.
\end{align*}
Substitution into \eqref{eq:ub} yields (recall that $\eta=L\in(0,1)$
and $\ell=\frac14$),
$$\mathcal E_{L,R}(u)\leq 2L^{-4/3}R\int_0^1g(t)\,dt+CRL^{-1/3}\,.$$
Since $u\in H^1_{\rm mag,0}(\mathcal S_R)$, then
$$\er\leq\mathcal E_{L,R}(u)\leq
2L^{-4/3}R\int_0^1g(t)\,dt+CRL^{-1/3}\,.$$ We divide by $R$ and use \eqref{eq:lb-er} to deduce that
$$E(L)\leq 2L^{-4/3}R\int_0^1g(t)\,dt+CL^{-1/3}\,.$$
\end{proof}

\section{Proof of Theorem~\ref{prop:HK+}}

Let $\ell\in(0,1)$ be a parameter {\bf independent} of $\kappa$.
Define the two sets,
$$\Omega_{\kappa,\ell}=\{\,x\in\Omega~:~|B_0(x)|<\frac1{b(\kappa)\kappa}\quad{\rm and}~{\rm dist}(x,\partial\Omega)>\ell\,\}\,,\quad
\Gamma_{\kappa,\ell}=\{x\in\Gamma~:~{\rm
dist}(x,\partial\Omega)>\ell\}\,.$$ Recall that $\Gamma=\{B_0=0\}$
and by Assumption~\ref{ass-B0}, $\Gamma\cap\partial\Omega$ is a
finite set. Thus, the area of $\Omega_{\kappa,\ell}$ and the length
of $\Gamma_{\kappa,\ell}$ satisfy, for $\kappa$ sufficiently large
and some constant $C>0$ (independent of $\kappa$ and $\ell$),
\begin{equation}\label{eq:appendix}
|\Omega_{\kappa,\ell}|\leq
\frac{C\varepsilon(\ell)}{b(\kappa)\kappa}\,,\quad
|\Gamma_{\kappa,\ell}|\leq C\varepsilon(\ell)\,,
\end{equation}
where $\varepsilon(\cdot)$ is a function independent of $\kappa$ and
satisfying
$\lim_{\ell\to0_+}\varepsilon(\ell)=0\,.$\\
The  standard proof of \eqref{eq:appendix} is left to the
reader. The estimate in \eqref{eq:appendix} is easier to
verify  under the additional assumption that $\Gamma$ and $\partial
\Omega$ intersect transversally, and in this case $\varepsilon(\ell)
= \ell$.
Note that $g(\cdot)$ vanishes in
$[1,\infty)$. Thus,
\begin{equation}\label{g=E} \int_\Omega g\big(b (\kappa) \, \kappa\,
|B_0(x)|\big)\,dx =\int_{\Omega_{\kappa,\ell}} g\big(b (\kappa) \,
\kappa\, |B_0(x)|\big)\,dx+\mathcal
O\left(\frac{\varepsilon(\ell)}{b(\kappa)\kappa}\right)\,.\end{equation}
Since $b(\kappa)\kappa\to +\infty$, then Assumption~\ref{ass-B0}
yields, for $\kappa$ sufficiently large,
\begin{equation}\label{eq:nablaB0}
\exists~C>0\,,\quad \forall~x\in\Omega_\kappa\,,\quad\Big|\,|\nabla B_0(x)|^{-1}-|\nabla
B_0(p( x))|^{-1}\,\Big|\leq \frac{C}{b(\kappa)\kappa}\,.
\end{equation}
Here, for $\kappa$ sufficiently large and for all
$x\in\Omega_{\kappa,\ell}$, the point $p(x)\in\Gamma$ is uniquely
defined by the relation
$${\rm dist}(x,\Gamma)={\rm dist}(x,p(x))\,.$$
The co-area formula yields,
$$\int_{\Omega_{\kappa,\ell}} g\big(b (\kappa) \, \kappa\, |B_0(x)|\big)\,dx=
\int_0^{\frac1{b(\kappa)\kappa}}\left(\int_{\{|B_0|=r\}\cap\Omega_{\kappa,\ell}}|\nabla
B_0(x)|^{-1}\,g\big(b(\kappa)\kappa \,r\big)\,ds\right)\,dr\,.$$
Thanks to \eqref{eq:nablaB0}, we get further,
\begin{multline*}
\int_{\Omega_{\kappa,\ell}} g\big(b (\kappa) \, \kappa\,
|B_0(x)|\big)\,dx\\=
\int_0^{\frac1{b(\kappa)\kappa}}\left(\int_{\{|B_0|=r\}\cap\Omega_{\kappa,\ell}}|\nabla
B_0(p(x))|^{-1}\,g\big(b(\kappa)\kappa
\,r\big)\,ds\right)\,dr+\mathcal
O\left(\frac1{\big(b(\kappa)\kappa\big)^2}\right)\,.\end{multline*}
Now, a simple calculation yields,
\begin{multline*}
\int_0^{\frac1{b(\kappa)\kappa}}\left(\int_{\{|B_0|=r\}\cap\Omega_{\kappa,\ell}}|\nabla
B_0(p(x))|^{-1}\,g\big(b(\kappa)\kappa
\,r\big)\,ds\right)\,dr\\
=\int_0^{\frac1{b(\kappa)\kappa}}\left(\int_{\{|B_0|=r\}\cap\Omega_{\kappa,\ell}}|\nabla
B_0(p(x))|^{-1}\,ds\right)g\big(b(\kappa)\kappa
\,r\big)\,dr\,,\end{multline*} and (using a simple analysis of the
arc-length measure in the curve $\{|B_0|=r\}$ and the assumption
that $\Gamma\cap\partial\Omega$ is a finite set),
\begin{multline*}
\forall~r\in\left(0,\frac1{b(\kappa)\kappa}\right)\,,\quad
\int_{\{|B_0|=r\}\cap\Omega_{\kappa,\ell}}|\nabla
B_0(p(x))|^{-1}\,ds\\
=\int_{\{|B_0|=0\}}|\nabla B_0(p(x))|^{-1}\,ds+\mathcal
O\big(\eta(\kappa)+\varepsilon(\ell)\big)\,,\quad(\kappa\to
\infty)\,,
\end{multline*}
where $\eta(\cdot)$ satisfies
$$\lim_{\kappa\to\infty}\eta(\kappa)=0\,.$$
As a consequence, we get the following formula,
$$
\int_{\Omega_{\kappa,\ell}} g\big(b (\kappa) \, \kappa\, |B_0(x)|\big)\,dx=
\int_{\Gamma}\left(\int_0^{\frac1{b(\kappa)\kappa}}\,g\big(b(\kappa)\kappa
\,r\big)\,dr\right)|\nabla
B_0(x)|^{-1}ds(x)+\mathcal O\left(\frac{\eta(\kappa)+\varepsilon(\ell)}{b(\kappa)\kappa}\right)
$$
A change of variable and Theorem~\ref{thm:HK} yield,
\begin{align*}
\int_0^{\frac1{b(\kappa)\kappa}}\,g\big(b(\kappa)\kappa
\,r\big)\,dr&=\frac1{b(\kappa)\kappa}\int_0^1g(t)\,dt\\
&=\frac1{2b(\kappa)\kappa}\left(L^{4/3}E(L)+\varepsilon_1(L)\right)\,,
\end{align*}
where $\lim_{L\rightarrow 0} \varepsilon_1(L) =0\,.$ \\
For $\kappa$ sufficiently large, we take
$$L=b(\kappa)|\nabla B_0(x)|\,,$$
and get,
$$
\int_{\Omega_{\kappa,\ell}} g\big(b (\kappa) \, \kappa\, |B_0(x)|\big)\,dx=
\frac1{2\kappa}\int_{\Gamma}|\nabla
B_0(x)|^{1/3}E\Big(b(\kappa)|\nabla B_0(x)|\Big)ds(x)+\mathcal O\left(\frac{ \lambda(\kappa)+\eta(\kappa)+\varepsilon(\ell)}{b(\kappa)\kappa}\right)\,,
$$
where $\lambda(\cdot)$ satisfies
$\displaystyle\lim_{\kappa\to\infty}\lambda(\kappa)=0$.
 Inserting
this into \eqref{g=E} and noticing that $\eta(\kappa)\to0$ as
$\kappa\to\infty$ and $\ell$ was arbitrary in $(0,1)$, then we get
the conclusion in Theorem~\ref{prop:HK+}.
\\~\\
{\bf Acknowledgements}\\
%
This work was done when the first author was Simons foundation
Visiting Fellow at the Isaac Newton Institute in Cambridge. The
support of the ANR project Nosevol is also acknowledged. The second
author acknowledges  financial support through a fund from Lebanese
University.

\end{document}